\documentclass[12pt]{article}
\usepackage[all]{xy}
\usepackage{amssymb,amsmath,amsthm,amscd,amsfonts}
\newtheorem{theorem}{Theorem}[section]
\newtheorem{lemma}[theorem]{Lemma}
\newtheorem{proposition}[theorem]{Proposition}
\newtheorem{corollary}[theorem]{Corollary}
\newtheorem{definition}[theorem]{Definition}
\newtheorem{example}[theorem]{Example}

\newtheorem{remark}[theorem]{Remark}
\numberwithin{equation}{section}
\def\im{{\rm im}}

\def\hom{{\rm Hom}}
\date{}
\title{\bf{Weak Factorization Systems and Fibrewise Regular Injectivity for Actions of Po-monoids on Posets}}
\author{{\bf Farideh Farsad}\and {\bf Ali Madanshekaf } \\
Department of Mathematics\\Faculty of Mathematics, Statistics and Computer Science\\
Semnan University\\ P. O. Box 35131-19111\\
Semnan\\
Iran\\ emails: faridehfarsad@yahoo.com\\
\qquad amadanshekaf@semnan.ac.ir}
\date{}
\begin{document}
\maketitle
\begin{abstract}
Let $S$ be a pomonoid. In this paper, {\bf Pos}-$S$, the category
of $S$-posets and $S$-poset maps, is considered. One of the main
aims of this paper is to draw attention to the notion of  weak
factorization systems in  {\bf Pos}-$S.$ We show that if the
identity element of $S$ is the bottom element, then
$(\mathcal{C_D}, \mathcal{E_S})$ is a weak factorization system in
{\bf Pos}-$S,$ where $\mathcal{C_D}$ and $\mathcal{E_S}$ are the
class of down-closed embedding $S$-poset maps and  the class of
all split $S$-poset epimorphisms, respectively. Among other
things, we use a fibrewise notion of complete posets in the
category {\bf Pos}-$S/B$ under a particular case where $B$ has
trivial action. We get a necessary condition for regular
injective objects in {\bf Pos}-$S/B$. Finally, we characterize
them under a spacial case, where $S$ is, a pogroup and conclude $(Emb,
Top)$ is a weak factorization system in {\bf Pos}-$S$. 
\end{abstract}
%----------------------------------------------------------------------------------------------
AMS {\it subject classification}: 06F05, 18A32, 18G05, 20M30, 20M50. \\
{\it key words}: $S$-poset, slice category, regular injectivity,
weak factorization system.
%-------------------------------------------section{}------------------------------------------
\section{Introduction}
%-----------------------------------------------------------------------------------------------
A comma category (a special case being a slice category) is a
construction in category theory. It provides another way of
looking at morphisms: instead of simply relating objects of a
category to one another, morphisms become objects in their own
right. This notion was introduced in 1963 by F. W. Lawvere,
although the technique did not become generally known until many
years later.

Injective objects with respect to a class $\mathcal{H}$ of
morphisms have been investigated for a long time in various categories.
Recently, injective objects in
slice categories ($\mathbf{C}/B$) have been investigated in detail
(see \cite{adam2, C.M}),
especially in relationship with weak factorization systems, a concept
used in homotopy theory, in particular for model categories.
More precisely, $\mathcal{H}$-injective objects in $\mathbf{C}/B$,
for any $B$ in $\mathbf{C}$,
form the right part of a weak factorization system that has morphisms
of $\mathcal{H}$ as the left part (see~\cite{adam2, B.R}).

In this paper, the notion of weak factorization system in {\bf
Pos}-$S$ is investigated. After some introductory notions in
section 1, we introduce in section 2, the notion of  weak
factorization system and state some related basic theorems. Also,
we give the guarantee about the existence of $(Emb, Emb^{\Box})$
as a weak factorization system in {\bf Pos}-$S$, where $Emb$ is
the class of all order-embeddings. We then find that every
$Emb$-injective object in {\bf Pos}-$S/B$ is split epimorphism. In
section 3, we continue studying $Emb$-injectivity using a
fibrewise notion of complete posets in the category {\bf
Pos}-$S/B$ under a particular case where $B$ has trivial action.

For the rest of this section, we give some preliminaries which we
will need in the sequel.

Given a category $\mathbf{C}$ and an object $B$ of $\mathbf{C}$,
one can construct the {\it slice category} $\mathbf{C}/B$ (read:
$\mathbf{C}$ over $B$): objects of $\mathbf{C}/B$ are morphisms of
$\mathbf{C}$ with codomain $B$, and morphisms in $\mathbf{C}/B$
from one such object $f: A\to B$ to another $g: C\to B$ are
commutative triangles in $\mathbf{C}$
$$
\xymatrix{A\ar[rr]^{h}\ar[dr]_{f}& &C\ar[dl]^{g}\\&B&}
$$
i.e, $gh=f$. The composition in $\mathbf{C}/B$ is defined from
the composition in $\mathbf{C}$, in the obvious way (paste
triangles side by side).

Let $\mathbf{C}$ be a category and $\mathcal{H}$ a class of its
morphisms. An object $I$ of  $\mathbf{C}$ is called
$\mathcal{H}$-{\it injective} if for each $\mathcal{H}$-morphism
$h :U\to V$ and morphism $u : U\to I$ there exists a morphism $s:
V\to I$ such that $sh = u$. That is, the following diagram is
commutative:
$$\xymatrix{U\ar[d]_{h}\ar[r]^{u} &I\\\ar@{-->}[ur]_{s} V &}$$
In particular, in the slice category $\mathbf{C}/B$, this means that,
$f : X\to B$ is $\mathcal{H}$-injective if, for any commutative diagram
$$
\xymatrix{U\ar[r]^{u}\ar[d]_{h} &X\ar[d]^f\\V\ar[r]_{v}&B}
$$
with $h\in \mathcal{H}$, there exists an arrow $s: V\to X$ such that $sh=u$ and $fs=\upsilon$.
$$
\xymatrix{U\ar[r]^{u}\ar[d]_{h}
&X\ar[d]^f\\V\ar@{-->}[ur]_{s}\ar[r]_{v}&B}
$$
The category $\mathbf{C}$ is said to have enough $\mathcal{H}$-injectives if
for every object $A$ of $\mathbf{C}$ there exists a morphism $A\to C$
in $\mathcal{H}$ where $C$ is an $\mathcal{H}$-injective object in $\mathbf{C}$.

Let $S$ be a monoid with identity 1. A (\emph{right})
$S$-\emph{act} or $S$-{\it set} is a set $A$ equipped with an
action $\mu: A\times S\rightarrow A,$ $(a,s)\mapsto as,$ such that
$a1=a$ and $a(st)=(as)t,$ for all $a\in A$ and $s,t\in S.$ Let
{\bf Act-$S$} denote the category of all $S$-acts with
action-preserving maps or $S$-maps. Clearly $S$ itself is an
$S$-act with its operation as the action. For instance, take any
monoid $S$ and a non-empty set $A$. Then $A$ becomes a right
$S$-act by defining $as = a$ for all $a\in A$, $s\in S$, we call
that $A$ an $S$-act with {\it trivial action} (see~\cite{K.K.M} or
\cite{M.T}).

$\bullet$ A poset is said to be {\it complete} if each of its
subsets has an infimum and a supremum.

Recall that a  {\it pomonoid} is a monoid with a partial order
$\leq$ which is compatible with the monoid operation: for $s,t,
s',t'\in S$, $s\leq t$, $s'\leq t'$ imply $ss'\leq tt'$. Now, let
$S$ be a pomonoid. A (\emph{right}) $S$-{\it poset} is a poset $A$
which is also an $S$-act whose action $\mu : A\times S\rightarrow
A$ is order-preserving, where $A\times S$ is considered as a poset
with componentwise order. The category of all $S$-posets with
action preserving monotone maps between them is denoted by {\bf
Pos}-$S$. Clearly $S$ itself is an $S$-poset with its operation as
the action. Also, if $B$ is a non-empty subposet of $A,$ then $B$
is called a {\it sub $S$-poset} of $A$ if $bs\in B$ for all $s\in
S$ and $b\in B$. Throughout this paper we deal with the pomonoid
$S$ and the category {\bf Pos}-$S$, unless otherwise stated. For
more information on $S$-posets see~\cite{F.M} or~\cite{F}.
%---------------------------------------------------------------------------------
\section{Weak Factorization System}
%---------------------------------------------------------------------------------
The concept of weak factorization systems plays an important role
in the theory of model categories. Formally, this notion
generalizes factorization systems by weakening the unique–
diagonalization property to the diagonalization property without
uniqueness. However, the basic examples of weak factorization
systems are fundamentally different from the basic examples of
factorization systems.

Now, we introduce from~\cite{adam2} the notion which we deal with in the paper.

{\bf Notation}: We denote by $\Box$ the relation {\it
diagonalization property} on the class of all morphisms of a category $\mathbf{C}$:
given morphisms $l: A\to B$ and $r: C\to D$ then
$$l\Box r$$
means that in every commutative square
$$
\xymatrix{A\ar[r]\ar[d]_{l}
&C\ar[d]^r\\B\ar@{-->}[ur]^{d}\ar[r]&D}
$$
there exists a diagonal $d: B\to C$ rendering both triangles
commutative. In this case $l$ is also said to have the {\it left
lifting property} with respect to $r$ (and $r$ to have the {\it
right lifting property} with respect to $l$).

Let $\mathcal H$ be a class of morphisms. We denote by
$${\mathcal{H}}^{\Box}=\{r | \ r\text{ has the right lifting property with
respect to each }l\in{\mathcal H}\}$$
and
$$~^{\Box}{\mathcal{H}} = \{ l | \ l \text{ has the left lifting
property with respect to each } r\in{\mathcal H}\}.$$
Let $\mathcal{H}_B$ be the class of those morphisms in
$\mathbf{C}/B$ whose underlying morphism in $\mathbf{C}$ lies in
$\mathcal{H}$. Now, $r: A\to B\in\mathcal{H}^{\Box}$ if and only
if $r$ is an $\mathcal{H}_B$-injective object in $\mathbf{C}/B$.
Dually, all morphisms in $~^{\Box}{\mathcal{H}}$ are
characterized by a projectivity condition in $\mathcal{H}_B$.

Recall from \cite{adam2} that a {\it weak factorization system} in a category is a pair
$(\mathcal{L},\mathcal{R})$ of morphism classes such that\\
(1) every morphism has a factorization as an
$\mathcal{L}$-morphism followed by an $\mathcal{R}$-morphism.\\
(2) $\mathcal{R}=\mathcal{L}^{\Box}$ and $\mathcal{L}=
~^{\Box}{\mathcal{H}}$.
\begin{remark}
{\rm If we replace ``$\Box$" by ``$\bot$" where ``$\bot$" is
defined via the {\it unique diagonalization property} (i.e., by
insisting that there exists precisely one diagonal), we arrive at
the familiar notion of a factorization system in a category.
Factorization systems are weak factorization systems. For
instance, let $\mathcal{E}$ be the class of all $S$-poset
epimorphisms. Then, by Theorem 1 of \cite{F.M} one can easily seen
that  $(\mathcal{E}, Emb)$ in {\bf Pos}-$S$ is a factorization
system.}
\end{remark}
Now, consider a functor $G: \mathcal{A}\to \mathcal{X}$. Recall
from \cite{adam2} that a source $(A\to A_i)_{i\in I}$ in
$\mathcal{A}$ is called $G$-initial provided that for each source
$(g_i: B\to A_i)_{i\in I}$ in $\mathcal{A}$ and each
$\mathcal{X}$-morphism $h: GB\to GA$ with $Gg_i=Gf_i h $ for each
$i\in I$, there exists a unique $\mathcal{A}$-morphism $\bar{h}:
B\to A$ in $\mathcal{A}$ with $G\bar{h}=h$ and $g_i=f_i\bar{h}$
for each $i\in I$.\\
Also, a source $(A\to A_i)_{i\in I}$ lifts a $G$-structured source
$(f_i: X\to GA_i)_{i\in I}$ provided that $G\bar{f_i}=f_i$ for
each $i\in I$.
\begin{definition}
A functor $G: \mathcal{A}\to \mathcal{X}$ in the category {\bf
Cat} (of all categories and functors) is {\it topological} if
every $G$-structured source $(X\to G(A_i))_{i\in I}$ has a unique
$G$-initial lift $(A\to A_i)_{i\in I}$.
\end{definition}
\begin{example}\label{6}
{\rm (1) In the category ${\bf Set}$ (of all sets and functions
between them) pair $(Mono, Epi)$ is a weak factorization system.
But $(Epi, Mono)$ is a factorization system in this category and
also in other categories, where $Mono$ is the class of all
monomorphisms and $Epi$ is the class of all epimorphisms in ${\bf
Set}$.\\
(2) The pair $(Full, Top)$ is a weak factorization system in the
category ${\bf Cat}$, where $Full$ is the class of those morphisms
in {\bf Cat} that are full and $Top$ is the class of those
morphisms in {\bf Cat} that are topological. \\
(3) In the category {\bf Pos}  of all posets with monotone maps,
the pair $(Emb, Top)$ is a weak factorization system, where $Emb$
is the class of all order-embeddings; that is, maps $f: A\to B$
for which $f(a)\leq f(a')$ if and only if $a\leq a'$, for all
$a,a'\in A$ and $Top$ is the class of all topological monotone
maps. For more details of proof see~\cite{adam2}.}
\end{example}

We record the following two results from \cite{adam2}, that will
be used in the sequel.
\begin{proposition}\label{wfs&enough injective}
Let $\mathbf{C}$ be a category and $\mathcal{H}$ a class of
morphisms closed under retracts in slice category $\mathbf{C}/B$
for
all objects $B$ of $\mathbf{C}$. Then the following conditions are equivalent:\\
(1) $(\mathcal{H} ,\mathcal{H}^{\Box})$ is a weak factorization
system.\\
(2) $\mathbf{C}/B$ has enough $\mathcal{H}_B$-injectives.
\end{proposition}

\begin{proposition}\label{Equi wfs}
Let $\mathbf{C}$ be a category. Then $(\mathcal{L},\mathcal{R})$
is a weak factorization system if and only if\\
(1) Any morphism $h\in \mathbf{C}$ has a factorization $h=gf$
with $f\in\mathcal{L}$ and $g\in\mathcal{R}$.\\
(2) For all $f\in\mathcal{L}$ and $g\in\mathcal{R}, f$ has the left
lifting property with respect to $g$.\\
(3) If $f: A\to B$ and $f^{\prime}: X\to Y$ are such that there
exist morphisms $\alpha: B\to Y$ and $\beta: A\to X$ then\\
(a) If $\alpha f\in\mathcal{L}$ and if $\alpha$ is a split
monomorphism then $f\in\mathcal{L}$.\\
(b) If $f^{\prime}\beta\in\mathcal{R}$ and if $\beta$ is split
epimorphism then $f^{\prime}\in\mathcal{R}$
\end{proposition}
If $f: A\to B$ and $g: A\to C$ are morphisms in a category
$\mathbf{C}$ such that there exist morphisms $\alpha: C\to B$ and
$\beta: B\to C$ with $\beta\alpha=1_C$, $\alpha g=f$ and $\beta
f=g$ then we say that $g$ is a $retract$ of $f$. In categorical
terms, $g$ is a retract of $f$ in the coslice category
$A/\mathbf{C}$.\\
Notice that in $3(a)$ above, $f$ is a retract of $\alpha f$ and that
all retracts can be written in this way. So this result is simply
saying that $\mathcal{L}$ is closed under retracts. Similarly $3(b)$
is equivalent to $\mathcal{R}$ being closed under retracts.

Recently, Bailey and Renshaw in \cite{B.R}, provide a number of
examples of weak factorization systems for $S$-acts such as the
following theorem. But first we need a definition.
\begin{definition}
{\rm We say that an $S$-act ($S$-poset) monomorphism $f: X\to Y$
is $unitary$ if $y\in \im (f)$ whenever $ys\in \im (f)$ and $s\in
S$. Clearly this is equivalent to saying that there exists an
$S$-act ($S$-poset) $Z$ such that $Y\cong X\dot\cup Z$ or in other
words, $\im (f)$ is a direct summand of $Y$.}
\end{definition}
\begin{theorem}
Let $S$ be a monoid and let $\mathcal{U}$ be the class of all
unitary $S$-monomorphisms and $\mathcal{E_S}$ be the class of all split $S$-act epimorphisms.
Then
$(\mathcal{U}, \mathcal{E_S})$ is a weak factorization system in {\bf Act-$S$}.
\end{theorem}

%-------------------------------------------------------------------------------------
\subsection{Weak Factorization Systems via Down-closed Embeddings}
Now, consider $Emb$ as the class of all embeddings of $S$-posets.
We try to provide a weak factorization system for {\bf Pos}-$S$
with $Emb$ as the left part. In this subsection, we consider
down-closed embeddings, as a subclass of $Emb$, and find a weak
factorization system in {\bf Pos}-$S$ with a condition on pomonoid
$S.$
\begin{definition}[\cite{S.M}]
{\rm A possibly empty sub $S$-poset $A$ of an $S$-poset $B$ is
said to be {\it down-closed}
 in $B$ if for each $a\in A$ and $b\in B$ with
$b\leq a$ we have $b\in A$. By a {\it down-closed embedding}, we
mean an embedding $f: A\to B$ such that $f(A)$ is a down-closed
sub $S$-poset of $B$.}
\end{definition}
Now, we prove the following crucial lemma in the category {\bf
Pos}-$S$.
\begin{lemma}\label{l:down-closed}
Let $S$ be a pomonoid whose identity element $e$ is the bottom element and
$f: X\to Y$ be an $S$-poset map. If $f$ is a down-closed embedding
then \im (f) is a direct summand of $Y$.
\end{lemma}
\begin{proof}
First we show that if $y\in \im (f)$ whenever $ys\in \im (f)$ and
$s\in S$. In fact, by hypothesis we have $e\leq s$ and we get $y\leq ys$, for
every $y\in Y$ and $s\in S$. Now, as $\im (f)$ is down-closed and $ys\in \im (f)$
then $y\in \im (f)$. Second, it is easy to see that
the above property of $f$ is equivalent to saying that there exists an $S$-poset
$Z$ (put $Z=Y\backslash \im (f))$ such that $Y\cong X\dot\cup Z$ or
in other words, $\im (f)$ is direct summand of $Y$.
\end{proof}
Let $\mathcal{C_D}$ denotes the class of down closed embedding
$S$-poset maps and let $\mathcal{E_S}$ denotes the class of all
split $S$-poset epimorphisms. We shall provide a weak
factorization system for {\bf Pos}-$S$ by these two classes. In
other words;
\begin{theorem}\label{wfs1}
Let $S$ be a pomonoid whose identity is the bottom element and $\mathcal{C_D}$
and $\mathcal{E_S}$ as above.
Then $(\mathcal{C_D}, \mathcal{E_S})$ is a weak factorization system in {\bf Pos}-$S$.
\end{theorem}
\begin{proof}
We must show all conditions of Proposition~\ref{Equi wfs}. For (1),
if $f: X\to Y$ is an $S$-poset map then we define the split
epimorphism $\bar{f}: X\dot\cup Y\to Y$ by $\bar{f}(x)=f(x)$,
$\bar{f}(y)=y$ for all $x\in X$, $y\in Y$. Let $i:
X\rightarrowtail X\dot\cup Y$ be the inclusion, it is easy to see
that $i$ is a down-closed embedding. Now, we have $f=\bar{f}i$ and
$i\in\mathcal{C_D}, \bar{f}\in\mathcal{E_S}$.\\
For condition (2), consider the commutative diagram
$$
\xymatrix{X\ar[r]^{{\rm {\it u}}}\ar@{{>}->}[d]_{f}
&C\ar[d]^g\\
Y\ar[r]_{v}&D}
$$
with $f\in\mathcal{C_D}$  and $g\in\mathcal{E_S}$. In view of
Lemma~\ref{l:down-closed}, without loss of generality we may
assume that $f$ is of the form  $f: X\rightarrowtail X\dot\cup Z.$
Then there exists $h: D\to C$ with $gh=1_D$ and so we can define
an $S$-poset map $k: X\dot\cup Z\to C$
by $k|_X=u$, $k|_Z=hv|_Z$ with the required property.\\
Finally, suppose that $f, f^{\prime}, \alpha, \beta$ are as
in condition (3) of Proposition \ref{Equi wfs}. If $f^{\prime}\beta$ is
a split epimorphism then clearly so is $f^{\prime}$. If $g: Y\to
A$ is such that $(f^{\prime}\beta)g=1_Y$ then
$f^{\prime}\in\mathcal{E_S}$ with splitting morphism $\beta g$.
Suppose that $\alpha f\in\mathcal{C_D}$ and if $\alpha$ is a split
monomorphism. Since $\alpha f$ is an embedding, $f$ is also. Now,
we show that $f(A)$ is a down-closed sub $S$-poset of $Y$. Given
$f(a)\in \im (f)$ for some $a\in A$ and $b\leq f(a)$. Then
$\alpha(b)\leq \alpha f(a)$. As $\alpha f$ is a
down-closed embedding we have $\alpha(b)\in \im(\alpha f)$. In
fact, $b\in \im (f)$ as $\alpha$ is a monomorphisms
(exactly one to one (see\cite{F.M})). Consequently,
$f\in\mathcal{C_D}$.
\end{proof}
Recall that each poset can be embedded (via an order-embedding)
into a complete poset, called the Dedekind-MacNeille completion.
In fact, given a poset $P$, its MacNeille completion is the poset
$\bar{P}$ consisting of all subsets $A$ of $P$ for which $LU(A) =
A,$ where $$U(A) = \{{x\in P: x\geq a, \ \forall a\in A}\}$$ and
$$LU(A)=\{ {y \in P : y\leq x,\forall x\in U(A)}\},$$ and the
embedding $\downarrow(-): P\to\bar{P}$ is given by $$a\mapsto
\downarrow(a) = \{{x\in P: x\leq a}\}$$ (see~\cite{B.B}).

Notice that in the category {\bf Pos}-$S/B$, regular monomorphisms
correspond to regular monomorphisms in {\bf Pos}-$S$ and these are
exactly order-embeddings (in {\bf Pos}-$S$) (see \cite{F.M,FF.M}).
We state the following theorem which gives us enough {\it
Emb}-injectivity property in {\bf Pos}-$S/B$. For details of the
proof see~\cite{FF.M}.
\begin{theorem}\label{enough Pos-S/B}
For an arbitrary $S$-poset $B$, the category {\bf Pos}-$S/B$ has
enough regular injectives. More precisely, each object $f : A\to
B$ in {\bf Pos}-$S/B$ can be regularly embedded into a regular
injective object $\pi_B^{\bar{A}^{(S)}}: \bar{A}^{(S)}\times B\to
B$ in {\bf Pos}-$S/B$ in which $\bar{A}^{(S)}$ is the set of all
monotone maps from $S$ into $\bar{A}$, with pointwise order and
the action is given by $(fs)(t) = f(st)$ for $s, t\in S$ and $f\in
\bar{A}^{(S)}$ and $\pi_B^{\bar{A}^{(S)}}: \bar{A}^{(S)}\times
B\rightarrow B$ is the second projection.
\end{theorem}
It is easy to show that the class $Emb$ closed under retracts in
{\bf Pos}-$S/B$. So by the Proposition \ref{wfs&enough injective}
 and above theorem, we can say that $(Emb,Emb^{\Box})$ is a weak
factorization system for {\bf Pos}-$S$. This implies that $Emb$ is
saturated (this means, every class in a category is closed under
pushouts, transfinite compositions and retracts (see~\cite{B.R})).

Up to now, we can not  succeed to determine if  there is a class
$\mathcal{R}$ such that $(Emb, \mathcal{R})$ is a weak
factorization system. However we do have:
\begin{proposition}\label{Emb in wfs}
Let $S$ be a pomonoid. Suppose $(Emb, \mathcal{R})$ is a weak
factorization system for {\bf Pos}-$S$. Then $\mathcal{R}\subseteq\mathcal{E_S}$.
\end{proposition}
\begin{proof}
Let $f: A\to B\in\mathcal{R}$ and consider the commutative diagram
$$
\xymatrix{A\ar[r]^{{\rm
 id}}\ar@{{>}->}[d]_{i}
&A\ar[d]^f\\A\dot\cup B\ar[r]_{\bar f}&B }
$$
where $i$ is the inclusion map and where $\bar f|_A=f$, $\bar
f|_B$=id. By assumption we have $\mathcal{R}=Emb^{\Box}$, so there
exists $h: A\dot\cup B\to A$ such that $fh=\bar f$ and $hi$=id.
This implies that $f$ is split epimorphism. Thus
$\mathcal{R}\subseteq\mathcal{E_S}$.
\end{proof}
%---------------------------------------------------------------%%
\section{\bf{Fibrewise Regular Injectivity of $S$-Poset Maps}}
%----------------------------------------------------------------%
In the last section, we deduced that every $Emb$-injective objects
in {\bf Pos}-$S/B$ is a split epimorphism (see Proposition
\ref{Emb in wfs}). In this section, we are going to characterize
them using a fibrewise notion of complete posets.

We recall that in the category {\bf Pos} of partially ordered sets
and monotone maps, a monotone map can be characterized as follows.
\begin{theorem}[\cite{tho}]\label{fibrewise Pos}
A monotone map $f: X\to B$ is $Emb$-injective in {\bf Pos}/$B$
if and only if it satisfies the following conditions:\\
{\em(I)} $f^{-1}(b)$ is complete poset, for every $b\in B.$\\
{\em(II)} $f$ is a fibration (that is, for every $x\in X$ and
$b\in B$ with $f(x)\leq b$, $\{x'\in f^{-1}(b)~|~x\leq x'\}$ has
a minimum element) and a cofibration (=dual of fibration).
\end{theorem}

The category of {\bf Pos}-$S$ is cartesian closed (see
\cite{F.M}). So by Theorem 1.2 from \cite{C.M} we have:

\begin{theorem}\label{characterizing of inj in slice cat}
Let $S$ be a pomonoid. Then $f: X\to B$ is a regular injective
object in {\bf Pos}-$S/B$ if and only if the following
two conditions are satisfied:\\
$(1)$ $\langle1_{X}, f\rangle: f\rightarrow\pi_B^{X}$ is a section
in {\bf Pos}-$S/B$ where $\pi_B^X :
X\times B\to B$ is the second projection.\\
$(2)$ The object $S_B(f)$ of sections of $f$ is a regular
injective object in {\bf Pos}-$S$.
\end{theorem}

Now, we supply a partial answer to the characterization of regular
injectivity in the category {\bf Pos}-$S/B$ in a special case,
when the $S$-poset $B$ has the trivial action.
\begin{remark}[\cite{FF.M}]\label{S_B(f)}
{\em For a pomonoid $S$, we know that {\bf Pos}-$S$ is cartesian
closed. Indeed, given two $S$-posets $A$ and $B$ the exponential
$B^A$ is given by $B^A=\hom(S\times A , B)$, the set of all
$S$-poset maps from the product $S$-poset $S\times A$ to $B$. Note
that the action on $S\times A$ operates on both components. This
set is an $S$-poset, with pointwise order and the action is given
by $(fs)(t,a)=f(st,a)$ (see~\cite{F.M,M.T}). Now, given $f:
X\rightarrow B$ in {\bf Pos}-$S/B$ we have
\begin{center}
$S_B(f)=\{h\in \hom(S\times B, X) \ \mid \  fh=\pi_B^S\}.$
\end{center}
Also, we consider the following embedding induced by fibres of
$f$:
\begin{center}
$m: S(f)\rightarrowtail\prod_{b\in B} f^{-1}(b)~~~{\rm with}~~~
m(h)=(h(b,s))_{{s\in S},~{b\in B}}$
\end{center}
and recall the following result from \cite{FF.M}.}
\end{remark}
\begin{proposition}\label{r.inj and fibre}
Let $S$ be a pomonoid. If $f: X\to B$ is a regular injective
object in the category {\bf Pos}-$S/B$, then:\\
$(1)~\langle1_{X}, f\rangle: X\to X\times B$ is a section in {\bf Pos}-$S/B$.\\
$(2)$ For every $b\in B$, the sub $S$-poset $f^{-1}(b)$ of $X$ is
regular injective object in {\bf Pos}-$S$, so it is a complete
poset.
\end{proposition}
In this section, we are going to give a new characterization of
injective objects in {\bf Pos}-$S/B$ that removes condition (1) of
the above proposition.
\begin{proposition}\label{section}
Let $S$ be a pomonoid and $f : X\to B$ be a $S$-poset map. If
$\langle 1_{X}, f\rangle: X\to X\times B$ is a section in {\bf
Pos}-$S/B$ then for $x\in X$ and $b\in B$ with $f(x)\leq b$,
$\{x'\in f^{-1}(b)~|~x\leq x'\}$ has a minimum element.
\end{proposition}
\begin{proof}
Let $r: X\times B\to X$ be a retraction of $\langle 1_{X},
f\rangle$ over B. For $x\in X$ and $b\in B$ with $f(x)\leq b$, let
$r(x, b)=x_b$ so we have
$$x=r(x, f(x))\leq r(x, b)=x_b$$ Also, take $x'$ in $f^{-1}(b)$ with
$x\leq x'$ then
$$x_b=r(x, b)\leq r(x', b)=x'$$
This means that $x_b$ is minimum in $\{x'\in f^{-1}(b)| x\leq
x'\}$.
\end{proof}
\begin{corollary}\label{fibrewise}
Let $S$ be a pomonoid and $f : X\to B$ be a regular injective
$S$-poset map. Then\\
{\rm (i)} For every $b\in B$, the sub $S$-poset $f^{-1}(b)$ of $X$
is regular injective object in {\bf Pos}-$S$, so it is a complete
poset.\\
{\rm (ii)} For $x\in X$ and $b\in B$ with $f(x)\leq b$, $\{x'\in
f^{-1}(b)~|~x\leq x'\}$ has a minimum element $x_b$ (also we have
the dual of this fact).
\end{corollary}
\begin{proof}
Applying Proposition \ref{r.inj and fibre} and the above
proposition we get the result.
\end{proof}

Now, we consider the category {\bf Pos}-$S$ as a sub category of
{\bf Cat}. On the other words, every $S$-poset is a category as a
poset and all action-preserving monotone maps are functors. So we
get the following result.

\begin{proposition}
Every regular injective object in {\bf Pos}-$S/B$ is a topological
$S$-poset map, considered as a functor.
\end{proposition}
\begin{proof}
First by Corollary \ref{fibrewise} and Theorem \ref{fibrewise
Pos}, one conclude that every regular injective object in {\bf
Pos}-$S/B$ is a regular injective object in ${\bf Pos}/B$. Then
part (3) from Example \ref{6}, says $(Emb)^\Box=Top$, as we
required.
\end{proof}
Next, we try to prove the converse of this fact.
\begin{proposition}\label{adjoint}
The functor $G_B: {\bf Pos}/B\to$ {\bf Pos}-$S/B$ that equips
monotone map in ${\bf Pos}/B$ with the trivial action and so a map
in {\bf Pos}-$S/B$, has a left adjoint.
\end{proposition}
\begin{proof}
Define the functor $H_B:$ {\bf Pos}-$S/B\to$ {\bf Pos}$/B$ given
by $H_B(h)=\bar{h}: A/\theta\to B$ with $\bar{h}([a])=h(a)$ for
every $h\in$ {\bf Pos}-$S/B$. Where the poset $A/\theta$ was
introduced in Theorem 12 from \cite{F.M} and $\bar{h}$ is a
monotone map. If $g: A\to C$ is an $S$-poset map over $B$, then
$H_B(g): A/\theta \to C/\theta$ defined by $g([a])=[g(a)]$ is a
well-defined monotone map over $B$. The unit of this adjunction
$$
\xymatrix{A\ar[rr]^{\eta_f}\ar[dr]_{f}&
&A/\theta\ar[dl]^{G_BH_B(f)}\\&B&}
$$
for an object $f: A\to B$ in {\bf Pos}-$S/B$, is the natural
$S$-poset over $B$. It is a universal arrow to $G_B$ because for a
given $S$-poset map
$$
\xymatrix{A\ar[rr]^{h}\ar[dr]_{f}& &P\ar[dl]^{G_B(l)}\\&B&}
$$
where $l: P\to B$ is a monotone map, we have a unique $S$-poset
map $\bar{h}$ as the following diagram
$$
\xymatrix{A/\theta\ar[rr]^{\bar{h}}\ar[dr]_{G_BH_B(f)}&
&P\ar[dl]^{G_B(l)}\\&B&}
$$
given by $\bar{h}([a])=h(a)$. By similar proof in Theorem 12 from
\cite{F.M} one can prove that $\bar{h}$ is a well defined
$S$-poset map. The above diagram is commutative, since for every
$[a]\in A/\theta$ we have:\\
$G_B(l)(\bar{h}[a])=G_B(l)(h(a))=f(a)=\bar{f}[a]=H_B(f)=G_BH_B(f)$.
\end{proof}
We state the following result from~\cite{B.B2} that will be used
in the sequel.
\begin{lemma}\label{5}
Let $F: \mathcal{C}\to \mathcal{D}$ and $G: \mathcal D\to
\mathcal{C}$ be two functors such that $F\dashv G$. Also, let
$\mathcal{M}$ and $\mathcal{N}$ be certain subclasses of
$\mathcal{C}$ and $\mathcal{D}$, respectively. If for all $f\in
\mathcal{M}$, $Ff\in \mathcal{N}$, then for any
$\mathcal{N}$-injective object $D$ of $\mathcal D$, $GD$ is an
$\mathcal{M}$-injective object of $\mathcal{C}$.
\end{lemma}
\begin{theorem}\label{top functor}
Let $S$ be a pogroup. Then all topological $S$-poset functors are
regular injective as $S$-poset map with trivial action.
\end{theorem}
\begin{proof}
By similar proof of the Theorem 4.6 in \cite{E.M.R}, we can show
that the functor $H_B:$ {\bf Pos}-$S/B\to$ {\bf Pos}$/B$ preserves
order-embeddings, these are the regular monomorphisms in two
category {\bf Pos}-$S/B$ and {\bf Pos}-$S$. Therefore, by Lemma
\ref{5} and the adjunction Proposition \ref{adjoint}, the functor
$G_B:$ {\bf Pos}$/B\to$ {\bf Pos}-$S/B$ preserves regular
injective objects. Since, part (3) from the Example \ref{6} says
$(Emb)^\Box=Top$ so we get the result.
\end{proof}
\begin{corollary}
Let $S$ be a pogroup. Then $(Emb, Top)$ is a weak factorization
system for $S$-posets with trivial action.
\end{corollary}
\begin{proof}
Every $S$-poset map $f: X\to B$ has a $(Emb, Top)$-factorization
$f: X\to Y\to B$ as a monotone map in {\bf Pos} (see Corollary 2.7
in \cite{adam2}) where $Y$ is an $S$-poset with trivial action.
\end{proof}
%-------------------------------------------------------------------------------%

\end{document}